\documentclass[11pt]{amsart}

\usepackage{amscd,amssymb,amsopn,amsmath,amsthm,mathrsfs,graphics,amsfonts,enumerate,verbatim,calc
}
\usepackage{bbm}
\usepackage[all,cmtip]{xy}
\usepackage[all]{xy}
\usepackage{tikz}
\usepackage{lscape}
\usepackage{enumitem}
\usetikzlibrary{matrix,arrows,decorations.pathmorphing}

\usepackage{scalerel}
\usepackage{stackengine,wasysym}
\usetikzlibrary {positioning}
\usetikzlibrary{patterns}
\usetikzlibrary{calc}
\definecolor {processblue}{cmyk}{0.96,0,0,0}

\usepackage{subfigure}

\usepackage{comment} 

\usepackage{ dsfont }

\usepackage{color}

\usepackage[OT2,OT1]{fontenc}
\newcommand\cyr{%
\renewcommand\rmdefault{wncyr}%
\renewcommand\sfdefault{wncyss}%
\renewcommand\encodingdefault{OT2}%
\normalfont
\selectfont}
\DeclareTextFontCommand{\textcyr}{\cyr}

\usepackage{amssymb,amsmath}

\DeclareFontFamily{OT1}{rsfs}{}
\DeclareFontShape{OT1}{rsfs}{n}{it}{<-> rsfs10}{}
\DeclareMathAlphabet{\mathscr}{OT1}{rsfs}{n}{it}

\numberwithin{equation}{section}
\newtheorem{theorem}{Theorem}[section]
\newtheorem{lem}[theorem]{Lemma}
\newtheorem{cor}[theorem]{Corollary}

\newtheorem{conjecturbe}[theorem]{Conjecture}
\newtheorem{quest}[theorem]{Question}
\newtheorem{prop}[theorem]{Proposition}

\theoremstyle{definition}
\newtheorem{defn}[theorem]{Definition}
\theoremstyle{remark}
\newtheorem{remark}[theorem]{Remark}

\newtheorem{example}[theorem]{Example}

\renewcommand{\tilde}{\widetilde}

\newcommand{\hit}{\operatorname{ht}}
\newcommand{\lk}{\operatorname{lk}}

\newcommand{\Spec}{\operatorname{Spec}}

\newcommand{\astar}{\operatorname{ast}}

\newcommand{\depth}{\operatorname{depth}}

\newcommand{\str}{\ensuremath{\operatorname{st}}}
\newcommand{\ho}{\tilde H}

\DeclareMathOperator{\core}{core}

\newcommand{\p}{\mathfrak{p}}

\newcommand{\set}[1]{\{ #1 \}}

\newcommand{\JLL}[1]{{\color{green}\sf #1}} 

\usepackage{xstring}

\newcommand{\drawsimplex}[4]{
	\IfEqCase{#1}{%
		{nw}{\draw [ultra thick, draw=black, draw opacity=1, pattern=north west lines, pattern color=yellow](#2) \foreach \i in #3{ -- (\i)} -- cycle;}
		{ne}{\draw [ultra thick, draw=black, draw opacity=1, pattern=north east lines, pattern color=yellow](#2) \foreach \i in #3{ -- (\i)} -- cycle;}
		{v}{\draw [ultra thick, draw=black, draw opacity=1, pattern=vertical lines, pattern color=yellow] (#2) \foreach \i in #3{ -- (\i)} -- cycle;}
 		{h}{\draw [ultra thick, draw=black, draw opacity=1, pattern=horizontal lines, pattern color=yellow] (#2) \foreach \i in #3{ -- (\i)} -- cycle;}
		{d}{\draw [ultra thick, draw=black, draw opacity=1, pattern=dots, pattern color=yellow] (#2) \foreach \i in #3{ -- (\i)} -- cycle;}
	}[\PackageError{drawsimplex}{Done messed up with #1}{}];

\draw [ultra thick, draw=black, draw opacity=1] \foreach \i in #4 \foreach \j in #4{(\i) -- (\j) };

}

\newcommand{\drawsimplexs}[4]{
	\IfEqCase{#1}{%
		{nw}{\draw [ultra thick, draw=black, draw opacity=1, pattern=north west lines, pattern color=blue](#2) \foreach \i in #3{ -- (\i)} -- cycle;}
		{ne}{\draw [ultra thick, draw=black, draw opacity=1, pattern=north east lines, pattern color=blue](#2) \foreach \i in #3{ -- (\i)} -- cycle;}
		{v}{\draw [ultra thick, draw=black, draw opacity=1, pattern=vertical lines, pattern color=blue] (#2) \foreach \i in #3{ -- (\i)} -- cycle;}
 		{h}{\draw [ultra thick, draw=black, draw opacity=1, pattern=horizontal lines, pattern color=blue] (#2) \foreach \i in #3{ -- (\i)} -- cycle;}
		{d}{\draw [ultra thick, draw=black, draw opacity=1, pattern=dots, pattern color=blue] (#2) \foreach \i in #3{ -- (\i)} -- cycle;}
	}[\PackageError{drawsimplex}{Done messed up with #1}{}];

\draw [ultra thick, draw=black, draw opacity=1] \foreach \i in #4 \foreach \j in #4{(\i) -- (\j) };

}

\newcommand{\drawintro}[1]{
\begin{tikzpicture}[scale=#1]
  \node[coordinate] at (0,0) (D){};
  \node at (D) [right] {$\mathbf{D}$};
  \node[coordinate] at (-1.3,0) (E){};
  \node at (E) [left] {$\mathbf{E}$};
  \node[coordinate] at (-0.3,2) (B){};
  \node at (B) [above] {$\mathbf{B}$};
  \node[coordinate] at (-0.3,-2) (G){};
  \node at (G) [below] {$\mathbf{G}$};
  \node[coordinate] at (1.3,-1.3) (H){};
  \node at (H) [below] {$\mathbf{H}$};
  \node[coordinate] at (1.3,1.3) (A){};
  \node at (A) [above] {$\mathbf{A}$};
  \node[coordinate] at (0.3,-1.2) (F){};
  \node at (F) [below] {$\mathbf{F}$};
  \node[coordinate] at (0.3,1.2) (C){};
  \node at (C) [above] {$\mathbf{C}$};
\drawsimplexs{v}{D}{{H,G}}{{D,G,H,F}}
  \node at ($.25*(D)+.25*(G)+.25*(H)+.25*(F)$) [above right] {$\mathbf{F_1}$};
\drawsimplexs{ne}{A}{{D,B}}{{A,C,D,B}}
  \node at ($.25*(A)+.25*(C)+.25*(D)+.25*(B)$) [below right] {$\mathbf{F_2}$};
\drawsimplexs{nw}{B}{{C,D,E}}{{D,C,B,E}}
  \node at ($.25*(D)+.25*(C)+.25*(B)+.25*(E)$) [left] {$\mathbf{F_3}$};
\drawsimplexs{h}{D}{{F,G,E}}{{D,G,E,F}}
  \node at ($.25*(D)+.25*(G)+.25*(E)+.25*(F)$) [left] {$\mathbf{F_4}$};
\end{tikzpicture}
}

\newcommand{\drawintroNone}[1]{
\begin{tikzpicture}[scale=#1]
  \node[coordinate] at (.5,-1) (F1){};
  \node at (F1) [below right] {$\mathbf{F_1}$};
  \node[coordinate] at (.5,1) (F2){};
  \node at (F2) [above right] {$\mathbf{F_2}$};
  \node[coordinate] at (-.5,.8) (F3){};
  \node at (F3) [above left] {$\mathbf{F_3}$};
  \node[coordinate] at (-.5,-.8) (F4){};
  \node at (F4) [below left] {$\mathbf{F_4}$};
  \node[coordinate] at (-0,2) (B){};
  \node at (B) [above] {};
  \node[coordinate] at (-0,-2) (G){};
  \node at (G) [below] {};
\drawsimplexs{v}{F1}{{F2,F3,F4}}{{F1,F2,F3,F4}}
\draw[black,thick,fill=black] (F1) circle [radius=2pt];
\draw[black,thick,fill=black] (F2) circle [radius=2pt];
\draw[black,thick,fill=black] (F3) circle [radius=2pt];
\draw[black,thick,fill=black] (F4) circle [radius=2pt];
\end{tikzpicture}
}

\newcommand{\drawintroNtwo}[1]{
\begin{tikzpicture}[scale=#1]
  \node[coordinate] at (.5,-1) (F1){};
  \node at (F1) [below right] {$\mathbf{F_1}$};
  \node[coordinate] at (.5,1) (F2){};
  \node at (F2) [above right] {$\mathbf{F_2}$};
  \node[coordinate] at (-.5,.8) (F3){};
  \node at (F3) [above left] {$\mathbf{F_3}$};
  \node[coordinate] at (-.5,-.8) (F4){};
  \node at (F4) [below left] {$\mathbf{F_4}$};
  \node[coordinate] at (-0,2) (B){};
  \node at (B) [above] {};
  \node[coordinate] at (-0,-2) (G){};
  \node at (G) [below] {};
\draw [ultra thick, draw=black]  (F1) -- (F4);
\draw [ultra thick, draw=black]  (F3) -- (F4);
\draw [ultra thick, draw=black]  (F3) -- (F2);
\draw[black,thick,fill=black] (F1) circle [radius=2pt];
\draw[black,thick,fill=black] (F2) circle [radius=2pt];
\draw[black,thick,fill=black] (F3) circle [radius=2pt];
\draw[black,thick,fill=black] (F4) circle [radius=2pt];
\end{tikzpicture}
}

\newcommand{\drawintroNthree}[1]{
\begin{tikzpicture}[scale=#1]
  \node[coordinate] at (.5,-1) (F1){};
  \node at (F1) [below right] {$\mathbf{F_1}$};
  \node[coordinate] at (.5,1) (F2){};
  \node at (F2) [above right] {$\mathbf{F_2}$};
  \node[coordinate] at (-.5,.8) (F3){};
  \node at (F3) [above left] {$\mathbf{F_3}$};
  \node[coordinate] at (-.5,-.8) (F4){};
  \node at (F4) [below left] {$\mathbf{F_4}$};
  \node[coordinate] at (-0,2) (B){};
  \node at (B) [above] {};
  \node[coordinate] at (-0,-2) (G){};
  \node at (G) [below] {};
\draw [ultra thick, draw=black]  (F1) -- (F4);
\draw [ultra thick, draw=black]  (F3) -- (F2);
\draw[black,thick,fill=black] (F1) circle [radius=2pt];
\draw[black,thick,fill=black] (F2) circle [radius=2pt];
\draw[black,thick,fill=black] (F3) circle [radius=2pt];
\draw[black,thick,fill=black] (F4) circle [radius=2pt];
\end{tikzpicture}
}

\newcommand{\drawintroNfour}[1]{
\begin{tikzpicture}[scale=#1]
  \node[coordinate] at (.5,-1) (F1){};
  \node at (F1) [below right] {$\mathbf{F_1}$};
  \node[coordinate] at (.5,1) (F2){};
  \node at (F2) [above right] {$\mathbf{F_2}$};
  \node[coordinate] at (-.5,.8) (F3){};
  \node at (F3) [above left] {$\mathbf{F_3}$};
  \node[coordinate] at (-.5,-.8) (F4){};
  \node at (F4) [below left] {$\mathbf{F_4}$};
  \node[coordinate] at (-0,2) (B){};
  \node at (B) [above] {};
  \node[coordinate] at (-0,-2) (G){};
  \node at (G) [below] {};
\draw[black,thick,fill=black] (F1) circle [radius=2pt];
\draw[black,thick,fill=black] (F2) circle [radius=2pt];
\draw[black,thick,fill=black] (F3) circle [radius=2pt];
\draw[black,thick,fill=black] (F4) circle [radius=2pt];
\end{tikzpicture}
}

\begin{document}
\title[Rank Selection and Depth Conditions for Balanced Simplicial Complexes]{Rank Selection and Depth Conditions for Balanced Simplicial Complexes}

\author{Brent Holmes}
\address{Department of Mathematics\\
University of Kansas\\
Lawrence, KS 66045-7523 USA}
\email{brentholmes@ku.edu}
\date{\today}

\author{Justin Lyle}
\address{Department of Mathematics\\
University of Kansas\\
Lawrence, KS 66045-7523 USA}
\email{justin.lyle@ku.edu}
\date{\today}



\begin{abstract}

 
We prove some new rank selection theorems for balanced simplicial complexes. Specifically, we prove that rank selected subcomplexes of balanced simplicial complexes satisfying Serre's condition $(S_{\ell})$ retain $(S_{\ell})$.  We also provide a formula for the depth of a balanced simplicial complex in terms of reduced homologies of its rank selected subcomplexes. By passing to a barycentric subdivision, our results give information about Serre's condition and the depth of any simplicial compex. Our results extend rank selection theorems for depth proved by Stanley, Munkres, and Hibi.

\end{abstract}

\maketitle

\section{Introduction}

Let $k$ be a field, $A=k[x_1,\dots,x_n]$, and $I$ a square-free monomial ideal in $A$.  The Stanley-Reisner correspondence associates to $R:=A/I$ a simplicial complex $\Delta$ whose topological and combinatorial properties capture the algebraic structure of $R$.  Exploiting this correspondence has been an active line of investigation in algebraic combinatorics over the past few decades.  Due to their combinatorial characterization (\cite[Theorem 1]{Re76}), Stanley-Reisner rings that are Cohen-Macaulay have received particular attention. 
However, the Cohen-Macaulay property is quite strong in this setting, and so there has been a focus in recent years on considering weaker algebraic properties such as Serre's condition $(S_{\ell})$ or bounds on $\depth R$ which still have interesting combinatorial ramifications.  For instance, even $(S_2)$ forces $\Delta$ to be pure, and $(S_{\ell})$ implies the $h$-vector of $R$ is nonnegative up to the $\ell$th spot (\cite{MT09});  see \cite{PF14} for a survey of related results.  The main purpose of this paper is to consider Serre's condition and the depth of Stanley-Reisner rings by studying balanced simplicial complexes.

A balanced simplicial complex $\Delta$ is a simplicial complex of dimension $d-1$, together with an ordered partition $\pi=(V_1,\dots,V_d)$ of the vertex set of $\Delta$ such that $|F \cap V_i| \le 1$ for every $F \in \Delta$ and every $i$.  To put it another way, the vertices of $\Delta$ are colored so that no face of $\Delta$ has more than one vertex of a given color.  The motivating example of a balanced simplicial complex is the order complex $\mathcal{O}(P)$ of a finite poset $P$, whose vertex set is $P$ and whose faces consist of all chains in $P$; we partition the vertices of $\mathcal{O}(P)$ by their height in $P$.  When $P$ is the face poset of a simplicial complex $\Delta$ (excluding the empty face), $\mathcal{O}(P)$ is nothing but the barycentric subdivision of $\Delta$, and it's well known that its geometric realization is homeomorphic to that of $\Delta$.  Thus we can study topological characteristics of any simplicial complex via the combinatorial structure of a balanced simplicial complex.  In particular, we may study homological properties such as the Cohen-Macaulay property and Serre's condition $(S_\ell)$, and numerical invariants such as $\depth$ in this manner.





Let $(\Delta,\pi)$ be a balanced simplicial complex of dimension $d-1$ with ordered partition $\pi=(V_1,\dots,V_d)$, and let $k[\Delta]$ denote its Stanley-Reisner ring over the field $k$.  If $S \subseteq [d]$, we let $\Delta_S$ be the subcomplex of $\Delta$ induced on $\bigcup_{i \in S} V_i$, and we refer to $\Delta_S$ as the $S$-rank selected subcomplex of $\Delta$.  It's often convenient to think about the ranks we remove rather than those we retain, and so we also set $\tilde{\Delta}_S:=\Delta_{[d]-S}$.  If $S=\{i\}$ is a singleton, we abuse notation and write $\Delta_i$ or $\tilde{\Delta}_i$, as appropriate. The so-called rank selection theorems of Stanley (\cite{St79}) and Munkres (\cite{Mu84}) show that homological properties often pass from $\Delta$ to $\Delta_S$. Specifically, we have the following: 

\begin{theorem}[{\cite{St79}}]\label{stanley}
Let $(\Delta,\pi)$ be a balanced simplicial complex. If $k[\Delta]$ is Cohen-Macaulay, then $k[\Delta_S]$ is Cohen-Macaulay for any $S \subseteq [d]$.
\end{theorem}


\begin{theorem}[{\cite{Mu84}}]\label{Munkres}
Let $(\Delta,\pi)$ be a balanced simplicial complex.  Then, for any $i \in [d]$, $\depth k[\tilde{\Delta}_{i}] \ge \depth k[\Delta]-1$.

\end{theorem}

As Serre's condition $(S_{\ell})$ generalizes the Cohen-Macaulay property, it is natural to consider if there is any extension of Theorem \ref{stanley} to $(S_{\ell})$. We prove this is indeed the case.

\begin{theorem}\label{balancedintro}

Let $(\Delta,\pi)$ be a balanced simplicial complex of dimension $d-1$.  If $k[\Delta]$ satisfies Serre's condition $(S_{\ell})$, then $k[\Delta_S]$ satisfies $(S_{\ell})$ for any $S \subseteq [d]$. 

\end{theorem}

If $P$ is a finite poset, we let $P_{>j}$ be the poset consisting of the elements of $P$ with height greater than $j$.  In the case $\Delta=\mathcal{O}(P)$ for a finite poset $P$, $\mathcal{O}(P_{>j})$ is the subcomplex of $\Delta$ with the bottom $j+1$ ranks removed.  For this case, one can nearly characterize $(S_{\ell})$ with the vanishing of reduced homologies of the $\mathcal{O}(P_{>j})$.

\begin{theorem}\label{serreintro}
Let $P$ be a finite poset.
\begin{enumerate}

\item If $k[\mathcal{O}(P)]$ satisfies $(S_{\ell})$, then $\tilde{H}_{i-1}(\mathcal{O}(P_{>j});k)=0$ whenever $i+j<d$ and $0 \le i<\ell$.

\item If $P$ is the face poset of a simplicial complex $\Delta$ and $\tilde{H}_{i-1}(\mathcal{O}(P_{>j});k)=0$ whenever $i+j<d$ and $0 \le i \le \ell$, then $k[\mathcal{O}(P)]$, and thus $k[\Delta]$, satisfies $(S_{\ell})$.

\end{enumerate}

\end{theorem}

It's natural to ask whether one can fully characterize $(S_{\ell})$ in this way i.e., whether the converse of $(1)$ or $(2)$ hold.  We provide examples (Examples \ref{noconverse1} and \ref{ex1}) that show this is not the case.

In general, equality need not hold in Theorem \ref{Munkres}; $\depth \tilde{\Delta}_{i}$ can be any value between $\dim \tilde{\Delta}_{i}$ and $\depth \tilde{\Delta}_{i}-1$.  However, we prove that one can often find a rank so that equality is achieved.

\begin{prop}\label{depthdropintro}

Let $(\Delta,\pi)$ be a balanced simplicial complex of dimension $d-1$, with ordered partition $\pi=(V_1,\dots,V_d)$.  If $\tilde{H}_{\depth k[\Delta]-1}(\Delta)=0$, then there is an $i \in [d]$ such that $\depth k[\tilde{\Delta}_{i}]=\depth k[\Delta]-1$.  

\end{prop}

Using Proposition \ref{depthdropintro}, we provide a formula for $\depth k[\Delta]$ (see Theorem \ref{balanceddepth}).





Finally, we 
provide a formula for sums of reduced Euler characteristics of links.  Our formula is analogous to those of \cite[Section 2 Lemma 1 (i)]{HI02} and \cite[Proposition 2.3]{Sw04}.

\begin{theorem}\label{fixedconjintro}

Suppose $\Delta$ is pure and let $P$ be the face poset of $\Delta$.  Write $\chi$ for Euler characteristic and $\tilde{\chi}$ for reduced Euler characteristic.  Then
\[\sum_{\substack{T \in \Delta \\ |T|=k}} \tilde{\chi}(\lk_{\Delta}(T))={\chi}(\mathcal{O}(P_{>k}))-{\chi}(\mathcal{O}(P_{>k-1})).\]

\end{theorem}

We now describe the structure of our paper.  In Section 2, we set notation and provide the algebraic and combinatorial background we appeal to throughout the paper.  In Section 3, we prove Theorems \ref{balancedintro} and \ref{serreintro}. Section 4 contains a proof of Proposition \ref{depthdropintro} as well as a formula for $\depth k[\Delta]$ (Theorem \ref{balanceddepth}).  In Section 5, we prove Theorem \ref{fixedconjintro} and provide an application to Gorenstein$^*$ complexes.  The last section discusses open problems related to this work and provides examples indicating the sharpness of our results.

\section{Background and Notation}

In this section we set notation and provide necessary background.  Once and for all, fix the base field $k$.  We let $\tilde{H}_i$ denote $i$th simplicial or singular homology, whichever is appropriate, always taken with respect to the field $k$.  We use $\chi$ for Euler characteristic and $\tilde{\chi}$ for reduced Euler characteristic.

Given a simplicial complex $\Delta$ we write $k[\Delta]$ for its Stanley-Reisner ring over $k$.  We write $V(\Delta)$ for the vertex set of $\Delta$, but, if $\Delta$ is clear from context, we generally write $V$ for $V(\Delta)$ and $n$ for $|V|$;  we set $A:=k[x_1,\dots,x_n]$. We write $f_i(\Delta)$ for the number of $i$-dimensional faces of $\Delta$, and $h_i(\Delta)$ for the $i$th entry of the $h$-vector of $\Delta$; so $h_i(\Delta)=\sum^k_{i=0} {d-i \choose k-i}(-1)^{k-i}f_{i-1}(\Delta)$.
We let $||\Delta||$ denote the geometric realization of $\Delta$.   We call $\Delta^{(k)} := \{ \sigma \in \Delta : \dim \sigma \le k \}$ the $k$-skeleton of $\Delta$.  

Given a subset $T \subseteq V(\Delta)$, we use $\Delta|_T:=\{\sigma \in \Delta \mid T \subseteq \sigma\}$ for the induced subcomplex of $\Delta$ on $T$. We may then define the star, the anti-star, and the link of $T$, respectively, as follows:
\begin{align*}
\str_\Delta T &:= \set{G \in \Delta \mid T \cup G \in \Delta} \\
\astar_\Delta T &:= \set{G \in \Delta \mid T \cap G = \varnothing} = \Delta|_{V-T} \\
\lk_\Delta T &:= \set{G \in \Delta \mid T \cup G \in \Delta \textrm{ and } T \cap G = \varnothing} = \str_\Delta T \cap \astar_\Delta T
\end{align*}

We note that $\str_{\Delta} T$ and $\lk_{\Delta} T$ are the void complex $\varnothing$ exactly when $T \notin \Delta$, and $\lk_{\Delta}(T)$ is the irrelevant complex $\{\varnothing\}$ exactly when $T$ is a facet of $\Delta$.  On the other hand, $\astar_{\Delta}(T)$ is nonempty as long as long as $T \ne V$. Of import, $\str_{\Delta}(T)$ is a cone over $\lk_{\Delta}(T)$ for any $T \in \Delta$, in particular is acyclic.  When $T=\{v\}$, we abuse notation and write $\str_{\Delta}(v)$, $\astar_{\Delta}(v)$, and $\lk_{\Delta}(v)$.

We say that $J \subseteq V(\Delta)$ is an independent set for $\Delta$ if $\{a,b\} \notin \Delta$ for any $a, b \in J$ with $a \ne b$.  Motivated by \cite{Hi91}, we say that $J \subseteq V(\Delta)$ is an excellent set for $\Delta$ if $J$ is an independent set for $\Delta$ and $J \cap F \ne \varnothing$ for every facet $F \in \Delta$.   When $\Delta$ is clear from context, we simply say that $J$ is an independent set or that $J$ is an excellent set, as appropriate.

The main computational tools of this paper are two exact sequences recorded in the following propositions:


\begin{prop}[{\cite[Lemma 4.2]{DD17}}]\label{exactlink}

Suppose $b$ is a non-isolated vertex of $\Delta$.  Then there is a Mayer-Vietoris exact sequence of the form
\[ \cdots \to \ho_{i}(\Delta) \to \ho_{i-1}(\lk_\Delta(b)) \to \ho_{i-1}(\astar_\Delta(b)) \to \ho_{i-1}(\Delta) \to \cdots \]

\end{prop}

\begin{prop}[{\cite[Proof of Lemma 4.3]{DD17}}]\label{exactindep}

Suppose $\{x\} \subsetneq J \subsetneq V$ and $J$ is an independent set.  Set $J'=J-\{x\}$.  Then there is a Mayer-Vietoris exact sequence of the form

\[\cdots \to \ho_{i}(\Delta) \to \ho_{i-1}(\astar_{\Delta}(J)) \to  \ho_{i-1}(\astar_{\Delta}(J')) \oplus \ho_{i-1} (\astar_{\Delta}(x)) \to \ho_{i-1}(\Delta) \to \cdots\]

\end{prop}

We also consider algebraic properties of $k[\Delta]$; one can see \cite{BH98} as a reference for this subject.  We use $\dim k[\Delta]$ for the Krull dimension of the ring $k[\Delta]$; so $\dim \Delta=\dim k[\Delta]-1$.  We write $d$ for $\dim k[\Delta]$ when $\Delta$ is clear from context.  By $\depth k[\Delta]$ we mean the depth of the $k$-algebra $k[\Delta]$; for a combinatorial characterization of $\depth$, see Proposition \ref{depthlink}.  We say $\Delta$ is Cohen-Macaulay whenever $k[\Delta]$ is Cohen-Macaulay, that is, if $\depth k[\Delta]=\dim k[\Delta]$.  Recall the following:

\begin{defn}

A commutative Noetherian ring $R$ satisfies \textit{Serre's Condition}, $(S_{\ell})$, if,  for all $\p \in \Spec R$, $\depth R_{\p} \ge \min\{\ell,\dim R_{\p}\}$.

\end{defn}

We say $\Delta$ satisfies $(S_{\ell})$ if $k[\Delta]$ does.  Every simplicial complex satisfies $(S_1)$, and a simplicial complex satisfies $(S_d)$ if and only if it is Cohen-Macaulay.
The following is an immediate consequence of Hochster's formula (\cite[Theorem 5.3.8]{BH98}) and gives a useful characterization of $\depth$ for Stanley-Reisner rings in terms of reduced homologies of links:

\begin{prop}\label{depthlink}

Let $\Delta$ be a simplicial complex.  Then $\depth k[\Delta] \ge t$ if and only if $\tilde{H}_{i-1}(\lk_{\Delta}(T))=0$ for all $T \in \Delta$ with $i+|T|<t$.

\end{prop}

The corresponding result for $(S_{\ell})$ can be found in \cite{Te07}:

\begin{prop}[\cite{Te07}]\label{serrelink}

Let $\Delta$ be a simplicial complex.  Then $\Delta$ satisfies $(S_{\ell})$ for $\ell \ge 2$ if and only if $\tilde{H}_{i-1}(\lk_{\Delta}(T))=0$ whenever $i+|T|<d$ and $0 \le i<\ell$.  In particular, $(S_{\ell})$ complexes are pure if $\ell \ge 2$.

\end{prop}

One can obtain similar characterizations for other algebraic properties of $k[\Delta]$.  Define $\core V(\Delta):=\{v \in V(\Delta) \mid \str_{\Delta}(v) \ne \Delta\}$ and set $\core \Delta:=\Delta|_{\core V(\Delta)}$.  We say that $\Delta$ is Gorenstein if the ring $k[\Delta]$ is Gorenstein; if, in addition, $\core \Delta=\Delta$, we say that $\Delta$ is Gorenstein$^*$.  One has the following, see \cite[Theorem 5.6.1]{BH98}:

\begin{theorem}\label{gorlink}

A simplicial complex $\Delta$ is Gorenstein$^*$ if and only if

\[\tilde{H}_{i-1}(\lk_{\Delta}(T)) \cong \begin{cases}  k & \mbox{ if } i=d-|T| \\ 0 & \mbox{ if } i \ne d-|T|   \end{cases}\]

\end{theorem}

Now, let $P$ be a finite poset. If $p \in P$, we let $\hit(p)$ denote the length of a longest chain $p_1 \prec p_2 \prec \cdots \prec p_i=p$ and let $\hit P:=\max\{\hit p \mid p \in P\}$.  We denote by $P_{>j}$ the poset obtained by restricting to elements $p \in P$ so that $\hit p>j$. The \textit{order complex} of $P$, denoted $\mathcal{O}(P)$, is the simplicial complex on $P$ consisting of all chains of elements in $P$.  Let $\mathcal{F}(\Delta)$ denote the face poset of $\Delta$.  We set $[\Delta]_{>j}:=\mathcal{O}(\mathcal{F}(\Delta)_{>j})$.  We note that when $j=0$, $[\Delta]_{>0}$ is the barycentric subdivision of $\Delta$.  The following is well known (see \cite[Corollary 5.7]{Gi77}, for example):

\begin{lem}\label{realization}

The realization $||\Delta||$ is homeomorphic to $||[\Delta]_{>0}||$. In particular, $\tilde{H}_{i}(\Delta) \cong \tilde{H}_{i}([\Delta]_{>0})$ for all $i$.

\end{lem}

We let $\rho:\Delta-\{\varnothing\} \to V([\Delta]_{>0})$ be the map which sends $T$ to itself viewed as a vertex of $[\Delta]_{>0}$.  

There are several advantages of working with $[\Delta]_{>k}$. For instance, the following result of \cite{DD17}:

\begin{lem}\label{isolinks}

Let $T \in \Delta$.  Then $[\lk_{\Delta}(T)]_{>0} \cong \lk_{[\Delta]_{>|T|-1}}(\rho(T))$ as simplicial complexes.  In particular, $\tilde{H}_i(\lk_{\Delta}(T)) \cong \tilde{H}_i(\lk_{[\Delta]_{>|T|-1}}(\rho(T)))$ for each $i$.

\end{lem}

\begin{defn}
A  balanced simplicial complex is a pair $(\Delta, \pi)$ satisfying:
\begin{enumerate}
\item $\Delta$ is $d-1$ dimensional simplicial complex on a vertex set $V$.
\item $\pi = (V_1,\dots,V_d)$ is an ordered partition of $V$.
\item For every facet $F \in \Delta$ and every $i \in [d]$, $|F \cap V_i| \le 1$.
\end{enumerate}
\end{defn}

Balanced simplicial complexes were introduced by Stanley in \cite{St79}.  One can find more information on balanced simplicial complexes in \cite{BS87, BG82, Ga80}; \cite{St96} gives a more modern treatment of the subject.  An important property of balanced simplicial complexes is that each $V_i$ is an independent set for $\Delta$, and, if $\Delta$ is pure, the $V_i$ are excellent sets for $\Delta$.  If $(\Delta,\pi)$ is a balanced simplicial complex with $\pi=(V_1,\dots,V_d)$, and if $S \subseteq [d]$, we define the \textit{$S$-rank selected subcomplex} of $\Delta$ to be the complex $\Delta_S:=\Delta|_{\bigcup_{i \in S} V_i}$; for notational convenience, we also set $\tilde{\Delta}_S=\Delta_{[d]-S}$.  If $(\Delta,\pi)$ is a balanced simplicial complex, we often suppress the ordered partition $\pi$ and simply refer to $\Delta$ as a balanced simplicial complex; in this case we always write $\pi=(V_1,\dots,V_d)$ for the corresponding ordered partition.

Now, let $P$ be a finite poset. If we set $V_i:=\{p \mid \hit(p)=i\}$ and $\pi=(V_1,\dots,V_{\hit P})$, then $(\mathcal{O}(P),\pi)$ is a balanced simplicial complex.  In particular, this means $[\Delta]_{>j}$ is always a balanced simplicial complex for any $j$.

Finally, we recall the higher nerve complexes of \cite{DD17}: 

\begin{defn}

Let $\{A_1,A_2,\dots,A_r\}$ be the collection of facets of $\Delta$.  The simplicial complex

\[N_i(\Delta):=\{F \subseteq [r] \colon |\bigcap_{j \in F} A_j| \ge i\}\]
is called the \textit{ith Nerve Complex} of $\Delta$.  We refer to the $N_i(\Delta)$ as the \textit{higher Nerve Complexes} of $\Delta$. We note that $N_0(\Delta)=2^{[r]}$ and $N_1$ is the classical nerve complex of $\Delta$.

\end{defn}

Higher nerve complexes capture important homological information about $\Delta$.  For our purposes, the important properties of higher nerve complexes can be summarized as follows: 
\begin{theorem}[{\cite[Theorems 1.3 and 2.8]{DD17}}]\label{centralnerve}
\
\begin{enumerate}[label=\arabic*\textup{)}]

\item $ \displaystyle \tilde{H}_{i-1}(N_{j+1}(\Delta))=0$ for $i+j>d$ and $1\leq j\leq d$.

\item $\displaystyle \depth k[\Delta]=\inf\{i+j \colon \tilde{H}_{i-1}(N_{j+1}(\Delta))=0\}$.

\item For $i \ge 0$,
\[\displaystyle f_i(\Delta)=\sum^{d-1}_{j=i} {j \choose i} \chi(N_{j+1}(\Delta)).\] 

\item $\tilde{H}_i([\Delta]_{>k}) \cong \tilde{H}_i(N_{k+1}(\Delta))$ for any $i$ and any $k$.

\end{enumerate}

\end{theorem}

\section{Rank Selection Theorems for Serre's Condition}\label{serresection}

In this section we prove some general statements and use them to derive Theorems \ref{balancedintro} and \ref{serreintro}.

\begin{lem}\label{excellentserre}

Suppose $J \subseteq V$ is excellent and $\Delta$ satisfies $(S_{\ell})$.  Set $\tilde{\Delta}:=\astar_{\Delta}(J)$. Then $\tilde{\Delta}$ satisfies $(S_{\ell})$.

\begin{proof}

We proceed by induction on $\ell$.  The claim is clear when $\ell=1$, since every simplicial complex satisfies $(S_1)$.  So, suppose we know the result for all $1 \le j \le \ell$ and suppose $\Delta$ satisfies $(S_{\ell+1})$.  Inductive hypothesis gives us that $\tilde{\Delta}$ satisfies $(S_{\ell})$, and we will show $\tilde{\Delta}$ satisfies $S_{\ell+1}$; the Lemma will then follow from induction.

By Proposition \ref{serrelink}, we have that $\tilde{H}_{i-1}(\lk_{\Delta}(T))=0$ whenever $i+|T|<d$ and $0 \le i \le \ell$, $\tilde{H}_{i-1}(\lk_{\tilde{\Delta}}(T))=0$ whenever $i+|T|<d-1$ and $0 \le i<\ell$, and we need only show that $\tilde{H}_{\ell-1}(\lk_{\tilde{\Delta}}(T))=0$ for all $T \in \tilde{\Delta}$ with $\ell+|T|<d-1$.

Pick $T \in \tilde{\Delta}$ such that $\ell+|T|<d-1$.  Let $\sigma \supseteq T$ be a facet of $\Delta$.  Since $J$ is excellent, there is a $b \in J \cap \sigma$, and thus $\{b\} \cup T \in \Delta$.
Since $b \notin T$, this means $b \in \lk_{\Delta}(T)$.  Note $T \cup \{b\}$ cannot be a facet of $\Delta$, since this would mean $|T|+1=d$, whilst $\ell+|T|<d-1$.  Set $S=J \cap V(\lk_{\Delta}(T))$; then we have $\lk_{\tilde{\Delta}}(T)=\astar_{\lk_{\Delta}(T)}(S)$.  By Proposition \ref{exactlink}, we have, for any $b \in S$, the exact sequence:
\[\tilde{H}_{\ell}(\astar_{\lk_{\Delta}(T)}(b)) \xrightarrow{i^*_b} \tilde{H}_{\ell}(\lk_{\Delta}(T)) \rightarrow \tilde{H}_{\ell-1}(\lk_{\lk_{\Delta}(T)}(b)) \rightarrow \tilde{H}_{\ell-1}(\astar_{\lk_{\Delta}(T)}(b)) \rightarrow \tilde{H}_{\ell-1}(\lk_{\Delta}(T))\]
where $i_b^*$ is the induced map coming from the inclusion $i_b:\astar_{\lk_{\Delta}(T)}(b) \hookrightarrow \lk_{\Delta}(T)$.  Since $\lk_{\lk_{\Delta}(T)}(b)=\lk_{\Delta}(T \cup \{b\})$ and since $\ell+|T|<d-1$, we have $\tilde{H}_{\ell-1}(\lk_{\lk_{\Delta}(T)}(b))=0$.  
Since $\tilde{H}_{\ell-1}(\lk_{\Delta}(T))=0$, we obtain $\tilde{H}_{\ell-1}(\astar_{\lk_{\Delta}(T)}(b))=0$ and that $i^*_b$ is surjective, from exactness.

Now, since $J$ is an independent set in $\Delta$, $S$ is an independent set in $\lk_{\Delta}(T)$.  We claim that $\tilde{H}_{\ell-1}(\astar_{\lk_{\Delta}(T)}(I))=0$ for any $\varnothing \subsetneq I \subseteq S$.  To see this, we induct on $|I|$.  Note that the claim is true when $|I|=1$, from above.  Now suppose the claim is true for every $I$ with $|I|=k$, and suppose we are given an $I$ with $|I|=k+1$.  Write $I=L \cup \{a\}$ so that $|L|=k$.  By Proposition \ref{exactindep} we have the exact sequence

\[\begin{tikzpicture}[descr/.style={fill=white,inner sep=1.5pt}]
        \matrix (m) [
            matrix of math nodes,
            row sep=1em,
            column sep=1.8em,
            text height=1.5ex, text depth=0.25ex
        ]
        {\ho_{\ell}(\astar_{\lk_{\Delta}(T)}(a)) \oplus \ho_{\ell}(\astar_{\lk_{\Delta}(T)}(L)) & \ho_{\ell}(\lk_{\Delta}(T)) \\
             \ho_{\ell-1}(\astar_{\lk_{\Delta}(T)}(I)) & \ho_{\ell-1}(\astar_{\lk_{\Delta}(T)}(a)) \oplus \ho_{\ell-1}(\astar_{\lk_{\Delta}(T)}(L)) \\
           };

        \path[overlay,->, font=\scriptsize,>=latex]
        (m-1-1) edge node[above]{$i^*_a-k^*$} (m-1-2)
        (m-1-2) edge[out=355,in=175] (m-2-1) 
        (m-2-1) edge (m-2-2);
      \end{tikzpicture}\]

\noindent where $k^*$ is the induced map coming from the inclusion $k:\astar_{\lk_{\Delta}(T)}(L) \hookrightarrow \lk_{\Delta}(T)$.

By inductive hypothesis, we have that $\ho_{\ell-1}(\astar_{\lk_{\Delta}(T)}(a)) \oplus \ho_{\ell-1}(\astar_{\lk_{\Delta}(T)}(L))=0$.  As we saw previously, $i^*_a$ is surjective so that $i^*_a-k^*$ is as well.  Thus we obtain $\tilde{H}_{\ell-1}(\astar_{\lk_{\Delta}(T)}(I))=0$ from exactness.  Therefore, induction gives us that $\tilde{H}_{\ell-1}(\astar_{\lk_{\Delta}(T)}(S))=\tilde{H}_{\ell-1}(\lk_{\tilde{\Delta}}(T))=0$, and thus, $\tilde{\Delta}$ satisfies $(S_{\ell+1})$.

\end{proof}

\end{lem}

Theorems \ref{balancedintro} and \ref{serreintro} $(1)$ now follow as quick consequences of Lemma \ref{excellentserre}:

\begin{theorem}\label{snapcasterbalance}
Let $\Delta$ be a balanced simplicial complex.  If $\Delta$ satisfies $(S_{\ell})$, then $\Delta_S$ satisfies $(S_{\ell})$ for any $S \subseteq [d]$.

\begin{proof}

The claim is clear when $\ell=1$.  When $\ell \ge 2$, $\Delta$ is pure, and the result follows by applying Lemma \ref{excellentserre} inductively on each $i \in [d]-S$.

\end{proof}

\end{theorem}

\begin{theorem}\label{serre1}

If $P$ is a finite poset satisfying $(S_{\ell})$, then $\ho_{i-1}(\mathcal{O}(P_{>j}))=0$ whenever $i+j<d$ and $0 \le i <\ell$.  In particular, if $\Delta$ is a simplicial complex satisfying $(S_\ell)$, then $\ho_{i-1}([\Delta]_{>j})=0$ whenever $i+j<d$ and $0 \le i <\ell$.

\begin{proof}

Suppose $P$ is $(S_{\ell})$.  By Theorem \ref{snapcasterbalance}, $\mathcal{O}(P_{>j})$ satisfies $(S_{\ell})$ for each $0 \le j \le d-1$.  In particular, $\tilde{H}_{i-1}(\mathcal{O}(P_{>j}))=0$ for $i<d-j$ and $0 \le i<\ell$.  It only remains to remark that if $\Delta$ is a simplicial complex satisfying $(S_{\ell})$, then, since $||\Delta|| \cong ||[\Delta]_{>0}||$ and since $(S_{\ell})$ is a topological property (\cite[Theorem 4.4 (d)]{Ya11}), $[\Delta]_{>0}$ satisfies $(S_{\ell})$.

\end{proof}

\end{theorem}

Remarkably, Theorem \ref{serre1} admits a partial converse (Theorem \ref{serreintro} $(2)$) when $P$ is the face poset of a simplicial complex.

\begin{theorem}\label{serre2}
If $\ho_{i-1}([\Delta]_{>j})=0$ whenever $i+j<d$ and $0 \le i \le \ell$, then $\Delta$ satisfies $(S_\ell)$.

\begin{proof}

We follow a similar approach to that of Lemma \ref{excellentserre}; we induct on $\ell$. The result is clear when $\ell=1$.  Suppose we know the result for $\ell$ and suppose $\ho_{i-1}([\Delta]_{>j})=0$ whenever $i+j<d$ and $0 \le i \le \ell+1$.  From induction hypothesis, we have that $\Delta$ satisfies $(S_{\ell})$.  Note that we assumed, in particular, that $\ho_0([\Delta]_{>j})=0$ whenever $j<d-1$.  Thus, no facet of $\Delta$ can have cardinality less than or equal to $d-1$; that is, $\Delta$ is pure.  Since $\Delta$ is $(S_{\ell})$, we have $\ho_{i-1}(\lk_{\Delta}(T))=0$ whenever $i+|T|<d$ and $0 \le i<\ell$, and we need only show that $\ho_{\ell-1}(\lk_{\Delta}(T))=0$ whenever $|T|<d-\ell$.  To see this, we proceed by induction on $|T|$.  When $|T|=0$, we have $\ho_{\ell-1}(\lk(T))=\ho_{\ell-1}(\Delta)=\ho_{\ell-1}([\Delta]_{>0})=0$.  Suppose $\ho_{\ell-1}(\lk(T))=0$ whenever $j=|T|<d-\ell$, and consider $T \in \Delta$ with $j+1=|T|<d-\ell$.

Letting $S=\{\rho(T) \mid T \in \Delta, |T|=j+1\}$ and writing $S=I \cup \{\rho(T)\}$, we have, by Proposition \ref{exactindep}, the exact sequence

\[\begin{tikzpicture}[descr/.style={fill=white,inner sep=1.5pt}]
        \matrix (m) [
            matrix of math nodes,
            row sep=1em,
            column sep=1.8em,
            text height=1.5ex, text depth=0.25ex
        ]
        {\ho_{\ell-1}([\Delta]_{>j+1}) & \ho_{\ell-1}(\astar_{[\Delta]_{>j}}(\rho(T))) \oplus \ho_{\ell-1}(\astar_{[\Delta]_{>j}}(I)) & \ho_{\ell-1}([\Delta]_{>j}) \\
           };

        \path[overlay,->, font=\scriptsize,>=latex]
        (m-1-1) edge (m-1-2)
        (m-1-2) edge (m-1-3);
        \end{tikzpicture}\]

Since $\ho_{\ell-1}([\Delta]_{>j+1})=0=\ho_{\ell-1}([\Delta]_{>j})$, we have $\ho_{\ell-1}(\astar_{[\Delta]_{>j}}(\rho(T))) \oplus \ho_{\ell-1}(\astar_{[\Delta]_{>j}}(I))=0$.  In particular, $\ho_{\ell-1}(\astar_{[\Delta]_{>j}}(\rho(T)))=0$.

As $\Delta$ is pure, $T$ is not a facet, and so $\rho(T)$ is a non-isolated vertex of $[\Delta]_{>j}$.  By Proposition \ref{exactlink}, we have the exact sequence
\[\ho_{\ell}(\astar_{[\Delta]_{>j}}(\rho(T))) \rightarrow \ho_{\ell}([\Delta]_{>j}) \rightarrow \ho_{\ell-1}(\lk_{[\Delta]_{>j}}(\rho(T))) \rightarrow \ho_{\ell-1} (\astar_{[\Delta]_{>j}}(\rho(T))) \rightarrow \ho_{\ell-1}([\Delta]_{>j})\]

Since $\ho_{\ell-1}(\astar_{[\Delta]_{>j}}(\rho(T)))=0=\ho_{\ell}([\Delta]_{>j})$, it follows that \(\ho_{\ell-1}(\lk_{[\Delta]_{>j}}(\rho(T)))=0=\) \(\ho_{\ell-1}(\lk (T))\), by Proposition \ref{depthlink}. Thus, $\Delta$ satisfies $(S_{\ell+1})$, and the result follows from induction.
\end{proof}

\end{theorem}

\begin{remark}
When $\ell =2$, the conclusion of Theorem \ref{serre1} is equivalent to $\ho_{0}([\Delta]_{>d-2})=0$, since, for a pure complex, connectivity of $[\Delta]_{>j}$ implies connectivity of $[\Delta]_{>j-1}$. 
\end{remark}

\begin{remark}

Since, by Theorem \ref{centralnerve} $(4)$, $\tilde{H}_{i-1}([\Delta]_{>j}) \cong \tilde{H}_{i-1}(N_{j+1}(\Delta))$ for any $i$ and $j$, Theorems \ref{serre1} and \ref{serre2} also serve as a version of Theorem \ref{centralnerve} $(2)$ for $(S_{\ell})$.

\end{remark}

\section{Depth of Rank Selected Subcomplexes}\label{depthsection}

The following lemma follows from \cite[Proposition 2.8]{Hi91} and a slightly weaker version can be found in \cite[Theorem 6.4]{Mu84}:

\begin{lem}\label{hibiindep}

Suppose $J$ is an independent set. Set $\tilde{\Delta}=\astar_{\Delta}(J)$.  Then $\depth \tilde{\Delta} \ge \depth \Delta-1$.

\end{lem}

We first provide a variation on this lemma:

\begin{lem}\label{depthindep}

Let $\depth \Delta=\ell$ and suppose $\tilde{H}_{\ell-1}(\Delta)=0$.  Choose $T \in \Delta$ of minimal cardinality such that $\tilde{H}_{\ell-|T|-1}(\lk_{\Delta}(T)) \ne 0$ (that such a $T$ exists follows from Proposition \ref{depthlink}).  Let $J$ be an independent set and suppose $T=T' \cup \{b\}$ with $b \in J$.  Set $\tilde{\Delta}=\astar_{\Delta}(J)$. Then $\tilde{H}_{\ell-|T'|-2}(\lk_{\tilde{\Delta}}(T')) \ne 0$.  In particular, $\depth\tilde{\Delta}=\ell-1$.  

\end{lem}

\begin{proof}

If $T$ is a facet of $\Delta$, then we have that $|T|=\ell$ by minimality, and, as $\lk_{\Delta}(T)=\lk_{\lk_{\Delta}(T')}(b)$, that $\{b\}$ is a facet of $\lk_{\Delta}(T')$.  By our minimality hypothesis, $\tilde{H}_0(\lk_{\Delta}(T'))=0$.  It follows that $\lk_{\Delta}(T')$ is a simplex with facet $\{b\}$, and so $\lk_{\tilde{\Delta}}(T')=\astar_{\lk_{\Delta}(T')}(b)=\{\varnothing\}$.  Thus $T'$ is a facet of $\tilde{\Delta}$, and so $\tilde{H}_{\ell-1-|T'|-1}(\lk_{\tilde{\Delta}}(T'))=\tilde{H}_{-1}(\lk_{\tilde{\Delta}}(T')) \ne 0$.

Otherwise, set $S=J \cap V(\lk_{\Delta}(T'))$ and note that $\lk_{\tilde{\Delta}}(T')=\astar_{\lk_{\Delta}(T')}(S)$.
Lemma \ref{exactlink} gives the following exact sequence 
\[\tilde{H}_{\ell-|T|}(\lk_{\Delta}(T')) \rightarrow \tilde{H}_{\ell-|T|-1}(\lk_{\lk_{\Delta}(T')}(b)) \rightarrow \tilde{H}_{\ell-|T|-1}(\astar_{\lk_{\Delta}(T')}(b)) \rightarrow \tilde{H}_{\ell-|T|-1}(\lk_{\Delta}(T'))\]

By minimality of $|T|$ and Proposition \ref{depthlink}, we have $\tilde{H}_{\ell-|T|}(\lk_{\Delta}(T'))=\tilde{H}_{\ell-|T|-1}(\lk_{\Delta}(T'))=0$.  Thus, $\tilde{H}_{\ell-|T|-1}(\lk_{\lk_{\Delta}(T')}(b)) \cong \tilde{H}_{\ell-|T|-1}(\astar_{\lk_{\Delta}(T')}(b))$.  But, $\lk_{\lk_{\Delta}(T')}(b)=\lk_{\Delta}(T' \cup \{b\})=\lk_{\Delta}(T)$, and so, in particular, $\tilde{H}_{\ell-|T|-1}(\astar_{\lk_{\Delta}(T')}(b)) \ne 0$.

But now, \cite[Lemma 4.3]{DD17} gives that $\tilde{H}_{i-|T|-1}(\astar_{\lk_{\Delta}(T')}(S)) \cong \bigoplus_{x \in S} \tilde{H}_{i-|T|-1}(\astar_{\lk_{\Delta}(T')}(x))$, in particular, is nonzero.  That $\depth\tilde{\Delta}=\ell-1$ now follows from Lemma \ref{hibiindep} and Proposition \ref{depthlink}.


\end{proof}

\begin{prop}\label{balanceddrop}

Let $\Delta$ be a balanced simplicial complex. Suppose $\tilde{H}_{\ell-1}(\Delta)=0$.  Then there exists an $i$ such that $\depth \astar_{\Delta}(V_i)=\ell-1$.

\end{prop}

\begin{proof}

This follows immediately from Lemma \ref{depthindep}

\end{proof}

Proposition \ref{depthdropintro} now follows immediately. 

With these results in hand, we now provide a formula for $\depth \Delta$.

\begin{theorem}\label{balanceddepth}
If $\Delta$ is a balanced simplicial complex, then
\[\depth \Delta=\min\{i+|S| \mid \tilde{H}_{i-1}(\tilde{\Delta}_S) \ne 0\}.\]

\end{theorem}

\begin{proof}

That 
\[\depth \Delta \le \min\{i+|S| \mid \tilde{H}_{i-1}(\tilde{\Delta}_{S}) \ne 0\}\]
follows at once from Lemma \ref{hibiindep}, so we need only concern ourselves with the reverse inequality.  We proceed by induction on $\depth \Delta$, noting that the claim is clear when $\depth \Delta=0$, that is, when $\Delta=\{\varnothing\}$. 
Suppose $\depth \Delta=\ell$.  The claim is clear if $\tilde{H}_{\ell-1}(\Delta) \ne 0$, so we may suppose this is not the case.  By Proposition \ref{balanceddrop}, there is an $i$ with $\depth \astar_{\Delta}(V_i)=\ell-1$.  From inductive hypothesis, we have $\ell-1=\min\{i+|S| \mid \tilde{H}_{i-1}(\astar_{\Delta}(V_i)_{[d]-S})\}$.
In particular, there is an $S \subseteq [d-1]$ with $\tilde{H}_{\ell-|S|-2}(\astar_{\Delta}(V_i))=\tilde{H}_{\ell-|S \cup \{i\}|-1}(\tilde{\Delta}_{|S \cup \{i\}|}) \ne 0$, and the result follows.

\end{proof}

\begin{cor}\label{posetdepth}
Let $P$ be a finite poset.  For any $S \subseteq \{1,\dots,\hit P\}$, let $\tilde{P}_S$ denote the poset obtained by restricting $P$ to elements whose height is not in $S$.  Then

\[\depth \mathcal{O}(P)=\min\{i+|S| \mid \tilde{H}_{i-1}(\mathcal{O}(\tilde{P}_S)) \ne 0\}.\]

In particular, for any simplicial complex $\Delta$, we can compute $\depth \Delta$ by taking $P$ to be the face poset of $\Delta$.

\end{cor}




\section{Euler Characteristics of Links and Truncated Posets}

We now shift our attention to Theorem \ref{fixedconjintro}.      

Similar to the proof of \cite[Lemma 1 (ii)]{HI02}, a simple counting argument shows

\[ \displaystyle \sum_{T \in F_k} f_{i-1}(\lk_{\Delta}(T))={i+k \choose k}f_{i+k-1}(\Delta).\]
As in \cite[Section 2 Lemma 1 (i)]{HI02} and \cite[Proposition 2.3]{Sw04}, one can combine this with Theorem 1.3 (3) to obtain a formula for $\displaystyle \sum_{\substack{T \in \Delta \\ |T|=k}} h_i(\lk_{\Delta}(T))$ in terms of Euler characteristics of higher nerves.  We follow a similar approach to obtain a particularly simple formula for $\displaystyle \sum_{\substack{T \in \Delta \\ |T|=k}} \tilde{\chi}(\lk_{\Delta}(T))$:

\begin{theorem}\label{fixedconj}

Suppose $\Delta$ is pure.  Then
\[\sum_{\substack{T \in \Delta \\ |T|=k}} \tilde{\chi}(\lk_{\Delta}(T))={\chi}([\Delta]_{>k})-{\chi}([\Delta]_{>k-1})\]

\begin{proof}

We make use of the following identity:

\[\sum^j_{i=0} (-1)^{i+1}{i+k \choose k}{j \choose i+k-1}=\begin{cases} -1 & j=k-1\\ 1 & j=k \\ 0 & j \ne k,k-1 \end{cases}\]

Set $F_k=\{T \in \Delta, |T|=k\}$.

Now,

\[
\begin{array}{rclcl}
\displaystyle \sum_{T \in F_k} \tilde{\chi}(\lk_{\Delta}(T)) &=& 
\displaystyle \sum^{d-k}_{i=0} \sum_{T \in F_k} (-1)^{i+1} f_{i-1}(\lk_{\Delta}(T)) \\
&=& \displaystyle \sum^{d-k}_{i=0} (-1)^{i+1}{i+k \choose k} f_{i+k-1}(\Delta) \\
&=& \displaystyle \sum^{d-k}_{i=0} \sum^{d-1}_{j=i+k-1}(-1)^{i+1}{i+k \choose k}{j \choose i+k-1} \chi(N_{j+1}(\Delta)) \\
&=& \displaystyle \sum^{d-1}_{j=0} \sum^{j}_{i=0}(-1)^{i+1}{i+k \choose k}{j \choose i+k-1} \chi(N_{j+1}(\Delta)) \\
&=& \displaystyle \chi(N_{k+1}(\Delta))-\chi(N_{k}(\Delta)) \\

\end{array}
\]

Th result then follows from Theorem \ref{centralnerve} (4).

\end{proof}

Note that, as long as $k \ne d$, $\displaystyle \sum_{T \in F_k} \tilde{\chi}(\lk_{\Delta}(T))={\chi}([\Delta]_{>k})-{\chi}([\Delta]_{>k-1})=\tilde{\chi}([\Delta]_{>k})-\tilde{\chi}([\Delta]_{>k-1})$.

\end{theorem}

\begin{cor}\label{eulercor}

Suppose $\Delta$ is pure.  Then

\[\sum^i_{k=j} \sum_{T \in F_k} \tilde{\chi}(\lk_{\Delta}(T))={\chi}([\Delta]_{>i})-{\chi}([\Delta]_{>j-1}).\]
In particular, \[\sum^i_{k=0} \sum_{T \in F_k} \tilde{\chi}(\lk_{\Delta}(T))={\chi}([\Delta]_{>i}).\]

\end{cor}

As an application, we provide a result analogous to those of sections \ref{serresection} and \ref{depthsection} for Gorenstein$^*$ complexes.

\begin{cor}\label{gorcor}

Suppose $\Delta$ is Gorenstein$^*$. Then
\[\dim_k \tilde{H}_{i-1}([\Delta]_{>j})=
\begin{cases} \displaystyle \dim_k \tilde{H}_{j-1}(\Delta^{(j-1)}) & \mbox{ if } i=d-j \\ 0 & \mbox{ if } i \ne d-j.    \end{cases}\]

The converse holds if $\lk_{\Delta}(T)$ is non-acyclic for each $T \in \Delta$. 

\begin{proof}

By Theorem \ref{centralnerve} $(4)$, $\tilde{H}_{i-1}([\Delta]_{>j}) \cong \tilde{H}_{i-1}(N_{j+1}(\Delta))$ for any $i$ and $j$.
Thus, by Theorems \ref{centralnerve} (1) and \ref{gorlink}, both conditions imply $\Delta$ is Cohen-Macaulay, in particular, that $\Delta^{(j-1)}$ is Cohen-Macaulay for every $j$ (\cite[Theorem 8]{Fr88}).  In this case, we have \[\dim_k \tilde{H}_{j-1}(\Delta^{(j-1)})=(-1)^j\tilde{\chi}(\Delta^{(j-1)})=\sum^j_{k=0} (-1)^{j-k}f_{k-1}(\Delta).\]

Suppose $\Delta$ is Gorenstein$^*$.  Then, by Theorem \ref{gorlink}
\[\tilde{H}_{i-1}(\lk_{\Delta}(T)) \cong \begin{cases} \displaystyle k & \mbox{ if } i=d-j \\ 0 & \mbox{ if } i \ne d-j    \end{cases}\]

Likewise, since $\Delta$ is Cohen-Macaulay, we have $\tilde{H}_{i-1}(N_{j+1}(\Delta))=0$ unless $i=d-j$ by Theorem \ref{centralnerve}.  By Corollary \ref{eulercor} we have

\[\sum^j_{k=0} \sum_{T \in F_k} \tilde{\chi}(\lk_{\Delta}(T))=\sum^j_{k=0} \sum_{T \in F_k} (-1)^{d-k-1}=\sum^j_{k=0} (-1)^{d-k-1}f_{k-1}(\Delta)=(-1)^{d-j-1}\dim_k \tilde{H}_{d-j-1}([\Delta]_{>j})\]
and the result follows.

Now suppose $\lk_{\Delta}(T)$ is non-acyclic for each $T \in \Delta$ and that \[\dim_k \tilde{H}_{i-1}([\Delta]_{>j})=\begin{cases} \displaystyle \sum^{j}_{k=0} (-1)^{j-k}f_{k-1}(\Delta) & \mbox{ if } i=d-j \\ 0 & \mbox{ if } i \ne d-j.    \end{cases}\].

Since $\Delta$ is Cohen-Macaulay, $\tilde{H}_{i-1}(\lk_{\Delta}(T))=0$ unless $i=d-|T|$.  Now we induct on $|T|$ to show that $\tilde{H}_{d-|T|-1}(\lk_{\Delta}(T)) \cong k$ for each $T$.  When $T=\varnothing$ we have $\dim \tilde{H}_{d-1}(\lk_{\Delta} T)=\dim \tilde{H}_{d-1}(\Delta) \cong \dim \tilde{H}_{d-1}([\Delta]_{>0})=f_{-1}(\Delta)=1$. Now suppose $\tilde{H}_{d-|T|-1}(\lk_{\Delta}(T)) \cong k$ whenever $|T|<j$.  Then
\[\displaystyle \sum^j_{k=0} \sum_{T \in F_k} \tilde{\chi}(\lk_{\Delta}(T))=\tilde{\chi}([\Delta]_{>j})=(-1)^{d-j-1}\dim_k \tilde{H}_{d-j-1}(N_{j+1}(\Delta))=\sum^j_{k=0} (-1)^{d-k-1}f_{k-1}(\Delta)\]
Similarly,
\[\displaystyle \sum^{j-1}_{k=0} \sum_{T \in F_k} \tilde{\chi}(\lk_{\Delta}(T))=\sum^{j-1}_{k=0}(-1)^{d-k-1}f_{k-1}(\Delta),\]
and thus
\[\sum_{T \in F_j} \tilde{\chi}(\lk_{\Delta}(T))=\sum_{T \in F_j} (-1)^{d-j-1}\dim_k \tilde{H}_{d-j-1}(\lk_{\Delta}(T))=(-1)^{d-j-1}f_{j-1}(\Delta).\]
Then \[\displaystyle \sum_{T \in F_j} \dim_k \tilde{H}_{d-j-1}(\lk_{\Delta}(T))=f_{j-1}(\Delta),\]
but, since $\lk_{\Delta}(T)$ is non-acyclic for each $T$, we must have $\dim_k \tilde{H}_{d-j-1}(\lk_{\Delta}(T))=1$ for each $T \in F_j$, by pigeonhole.  The result now follows from induction.

\end{proof}

\end{cor}

\begin{remark}

We claim the result of Corollary \ref{gorcor} is analogous to those of Sections \ref{serresection} and \ref{depthsection}, but this is perhaps not obvious. To see this, note that $\dim_k \tilde{H}_{i-1}(\Delta^{(j-1)})=\dim_k \tilde{H}_{i-1}(P^{-1}_{>d-j})$ where $P$ is the face poset of $\Delta$ (excluding $\varnothing$).  In essence, our result says that, when $\Delta$ is Gorenstein$^*$, removing $j$ ranks from the bottom of $P$ gives the same homologies as removing $d-j$ ranks from the top, though they are in different degrees.

\end{remark}

\section{Open Problems and Examples}\label{end}

We say that $A \subseteq \Delta$ is independent if $\sigma \cup \tau \notin \Delta$ for all $\sigma, \tau \in A$ with $\sigma \ne \tau$.  We say that $A$ is excellent if, additionally, for every facet $F$ of $\Delta$, $F \supseteq \sigma$ for some (necessarily unique) $\sigma \in A$.  Note that $J=\{v_1,\dots,v_m\} \subseteq V$ is independent (resp. excellent) if and only if $\{\{v_1\},\dots,\{v_m\}\}$ is an independent (resp. excellent) subset of $\Delta$.  If $A \subseteq \Delta$ is independent, we set
\[\Delta_A:=\Delta-\{\sigma \in \Delta \mid \sigma \supseteq \tau \mbox{ for some } \tau \in A\}.\]
If $A=\{\{v_1\},\dots,\{v_m\}\}$ where $J=\{v_1,\dots,v_m\} \subseteq V$ is independent, then $\Delta_A=\astar_{\Delta}(J)$.
Essentially the same argument as \cite[Proposition 2.8]{Hi91} shows the following extension of Lemma \ref{depthindep}:

\begin{prop}\label{genhibi}

Suppose $A \subseteq \Delta$ is independent.  Then $\depth \Delta_A \ge \depth \Delta-1$.

\end{prop}

We conjecture a similar extension of Lemma \ref{excellentserre}.

\begin{conjecturbe}

Suppose $A \subseteq \Delta$ is excellent.  If $\Delta$ satisfies $(S_{\ell})$, then $\Delta_A$ satisfies $(S_{\ell})$.

\end{conjecturbe}

\begin{remark}

If $A$ is independent and $\ell \ge 2$, the conclusion can only hold if $A$ is excellent, since $(S_2)$ complexes are pure.  Similar to Proposition \ref{genhibi}, one can modify the argument of \cite[Proposition 2.8]{Hi91} to show that $\Delta_A$ satisfies $(S_{\ell-1})$ whenever $\Delta$ satisfies $(S_{\ell})$ and $A$ is excellent.  However, as in the proof of Theorem \ref{serre1}, one often needs to cut away excellent subsets inductively, and, for this purpose, $(S_{\ell-1})$ is not generally good enough; in particular, we cannot conclude anything  when $\Delta$ only satisfies $(S_2)$.  A positive answer to this conjecture would allow one to extend Theorem \ref{balancedintro} to balanced complexes of a more general type, along the lines of \cite[Section 3]{Hi91}.

\end{remark}

The following examples show the converses of Theorems \ref{serre1} and \ref{serre2} do not hold, even for face posets of simplicial complexes:

\begin{example}\label{noconverse1}
Consider the complex $\Delta_1$ with facets: \[\{4,5,6\},\{1,5,6\},\{1,3,5\},\{2,3,6\},\{2,5,6\},\{2,4,6\}.\]  This complex is not $(S_2)$ but has $\ho_{i-1}([\Delta_1]_{>j})=0$ for all $i,j$ with $i+j<d$ and $0 \le i < 2$.
\end{example}

\begin{example}\label{ex1}
Consider the complex $\Delta_2$ with facets:
\[\{4,5,6\},\{3,5,6\},\{2,3,5\},\{2,3,4\},\{1,3,4\},\{2,4,6\}.\]
This complex is $(S_2)$ but $\ho_{1}([\Delta_2]_{>0})$ is non-trivial.
\end{example}

In fact, $\tilde{H}_{i-1}([\Delta_1]_{>j}) \cong \tilde{H}_{i-1}([\Delta_2]_{>j})$ for every $i$ and every $j$.  Since $\Delta_2$ is $(S_2)$ and $\Delta_1$ is not, this shows that $(S_2)$ cannot be determined in general by reduced homologies of the $[\Delta]_{>j}$.  Further, Example \ref{ex1} is Buchsbaum while Example \ref{noconverse1} is not, so Buchsbaum cannot be determined either.  In a similar fashion, the following example shows that Gorenstein cannot be detected in general.

\begin{example}

Let $\Gamma_1$ be the complex with facets
\[\{2,3,4\},\{1,3,4\},\{1,2,5\},\{2,3,5\},\{1,2,4\},\{1,3,5\}\] and $\Gamma_2$ the complex with facets
\[\{1,2,3\},\{1,2,4\},\{1,3,4\},\{2,3,4\},\{1,2,5\},\{1,3,5\}.\]
Then $[\Gamma_1]_{>j}$ and $[\Gamma_2]_{>j}$ have isomorphic homologies for each $j$, but $\Gamma_1$ is Gorenstein whilst $\Gamma_2$ is not (it is not even $2$-Cohen-Macaulay).

\end{example}  The above discussion leads us to ask the following general question:

\begin{quest}

In addition to the reduced homologies of the $[\Delta]_{>j}$, what information does one need to determine if a simplicial complex satisfies homological conditions such as $(S_{\ell})$, Buchsbaum, or Gorenstein?

\end{quest}

\section*{Acknowledgments}

We would like to thank Joseph Doolittle, Ken Duna, and Bennet Goeckner; we extend a special thanks to our advisor Hailong Dao and to Jonathan Monta\~{n}o for providing useful feedback on earlier drafts on the paper.  We would also like to thank Vic Reiner for helpful discussions, which led us to examine balanced simplicial complexes.

\bibliographystyle{amsalpha}
\bibliography{mybib}

\newcommand{\etalchar}[1]{$^{#1}$}
\providecommand{\bysame}{\leavevmode\hbox to3em{\hrulefill}\thinspace}
\providecommand{\MR}{\relax\ifhmode\unskip\space\fi MR }
\providecommand{\MRhref}[2]{%
  \href{http://www.ams.org/mathscinet-getitem?mr=#1}{#2}
}
\providecommand{\href}[2]{#2}
\begin{thebibliography}{PSFTY14}

\bibitem[BGS82]{BG82}
Anders Bj{\"o}rner, Adriano~M Garsia, and Richard~P Stanley, \emph{An
  introduction to cohen-macaulay partially ordered sets}, Ordered sets,
  Springer, 1982, pp.~583--615.

\bibitem[BH98]{BH98}
Winfried Bruns and H.~J{\"u}rgen Herzog, \emph{Cohen-{M}acaulay rings},
  Cambridge University Press, 1998.

\bibitem[BSF87]{BS87}
Anders Bj{\"o}rner, Richard Stanley, and Peter Frankl, \emph{The number of
  faces of balanced cohen-macaulay complexes and a generalized macaulay
  theorem}, Combinatorica \textbf{7} (1987), no.~1, 23--34.

\bibitem[DDD{\etalchar{+}}17]{DD17}
Hailong Dao, Joseph Doolittle, Ken Duna, Bennet Goeckner, Brent Holmes, and
  Justin Lyle, \emph{Higher nerves of simplicial complexes}, preprint.

\bibitem[FrÃ90]{Fr88}
Ralf Fröberg, \emph{On {S}tanley-{R}eisner rings}, Banach Center Publications
  \textbf{26} (1990), no.~2, 57--70 (eng).

\bibitem[Gar80]{Ga80}
Adriano~M Garsia, \emph{Combinatorial methods in the theory of cohen-macaulay
  rings}, Advances in Mathematics \textbf{38} (1980), no.~3, 229--266.

\bibitem[Gib77]{Gi77}
P.~J. Giblin, \emph{Graphs, surfaces and homology}, Springer Netherlands, 1977.

\bibitem[Hib91]{Hi91}
Takayuki Hibi, \emph{Quotient algebras of {S}tanley-{R}eisner rings and local
  cohomology}, J. Algebra \textbf{140} (1991), no.~2, 336--343. \MR{1120426}

\bibitem[HN02]{HI02}
Patricia Hersh and Isabella Novik, \emph{A short simplicial {$h$}-vector and
  the upper bound theorem}, Discrete Comput. Geom. \textbf{28} (2002), no.~3,
  283--289. \MR{1923232}

\bibitem[MT09]{MT09}
Satoshi Murai and Naoki Terai, \emph{{$h$}-vectors of simplicial complexes with
  {S}erre's conditions}, Math. Res. Lett. \textbf{16} (2009), no.~6,
  1015--1028. \MR{2576690}

\bibitem[Mun84]{Mu84}
James~R. Munkres, \emph{Topological results in combinatorics}, Michigan Math.
  J. \textbf{31} (1984), no.~1, 113--128. \MR{736476}

\bibitem[PSFTY14]{PF14}
M.~R. Pournaki, S.~A. Seyed~Fakhari, N.~Terai, and S.~Yassemi, \emph{Survey
  article: {S}implicial complexes satisfying {S}erre's condition: a survey with
  some new results}, J. Commut. Algebra \textbf{6} (2014), no.~4, 455--483.
  \MR{3294858}

\bibitem[Rei76]{Re76}
Gerald~Allen Reisner, \emph{Cohen-{M}acaulay quotients of polynomial rings},
  Advances in Math. \textbf{21} (1976), no.~1, 30--49. \MR{0407036}

\bibitem[Sta79]{St79}
Richard~P Stanley, \emph{Balanced cohen-macaulay complexes}, Transactions of
  the American Mathematical Society \textbf{249} (1979), no.~1, 139--157.

\bibitem[Sta96]{St96}
Richard~P. Stanley, \emph{Combinatorics and commutative algebra}, second ed.,
  Progress in Mathematics, vol.~41, Birkh\"auser Boston, Inc., Boston, MA,
  1996. \MR{1453579}

\bibitem[Swa05]{Sw04}
E.~Swartz, \emph{Lower bounds for {$h$}-vectors of {$k$}-{CM}, independence,
  and broken circuit complexes}, SIAM J. Discrete Math. \textbf{18} (2004/05),
  no.~3, 647--661. \MR{2134424}

\bibitem[Ter07]{Te07}
Naoki Terai, \emph{Alexander duality in {S}tanley-{R}eisner rings}, Affine
  algebraic geometry, Osaka Univ. Press, Osaka, 2007, pp.~449--462.
  \MR{2330484}

\bibitem[Yan11]{Ya11}
Kohji Yanagawa, \emph{Dualizing complex of the face ring of a simplicial
  poset}, J. Pure Appl. Algebra \textbf{215} (2011), no.~9, 2231--2241.
  \MR{2786613}

\end{thebibliography}

\end{document}